\documentclass{amsart}
\usepackage[utf8]{inputenc}

\usepackage{tikz} 
\usetikzlibrary{cd}
\usetikzlibrary{patterns,decorations}
\usepackage[utf8]{inputenc}
\usepackage{amsmath, amssymb}
\usepackage{amsthm, mathtools}
\usepackage[mathcal]{euscript}

\usepackage{dsfont}
\usepackage{float}
\usepackage{enumitem}
\usepackage{stmaryrd}
\usepackage{hyperref}

\usetikzlibrary{matrix}

\makeatletter
\@namedef{subjclassname@2020}{\textup{2020} Mathematics Subject Classification}
\makeatother

\definecolor{darkblue}{rgb}{0,0,0.6}
\usepackage[capitalize,noabbrev]{cleveref}

\newcommand{\FF}{\mathbb{F}}
\newcommand{\ZZ}{\mathbb{Z}}
\newcommand{\QQ}{\mathbb{Q}}

\newcommand{\fn}{{\mathfrak{n}}}
\newcommand{\fm}{{\mathfrak{m}}}
\newcommand{\fq}{{\mathfrak{q}}}
\newcommand{\pos}[2][k]{#1\llbracket{#2}\rrbracket}
\newcommand{\xra}{\xrightarrow}
\newcommand{\FreeCovers}{\mathcal{F}}
\newcommand{\filt}[1]{\Phi^{#1}}
\newcommand{\wt}{\widehat}

\DeclareMathOperator{\id}{id}
\DeclareMathOperator{\edim}{edim}
\DeclareMathOperator{\gr}{gr}
\DeclareMathOperator{\Hom}{Hom}
\DeclareMathOperator{\Ann}{Ann}

\DeclareMathOperator{\Tor}{Tor}
\DeclareMathOperator{\aut}{aut}
\DeclareMathOperator{\im}{im}
\DeclareMathOperator{\rank}{rank}
\DeclareMathOperator{\ord}{ord}

\DeclareMathOperator{\depth}{depth}

\numberwithin{equation}{section}

\newtheorem{theorem}[equation]{Theorem}

\newtheorem{corollary}[equation]{Corollary}
\newtheorem{lemma}[equation]{Lemma}
\newtheorem{proposition}[equation]{Proposition}

\theoremstyle{definition}
\newtheorem{example}[equation]{Example}
\newtheorem{definition}[equation]{Definition}
\newtheorem{remark}[equation]{Remark}

\newtheoremstyle{TheoremNum}
    {}{}       
    {\itshape}                 
        {}                              
        {\bfseries}                     
        {.}                             
        { }                             
        {\thmname{#1}\thmnote{ \bfseries #3}}
\theoremstyle{TheoremNum}

\title[Automorphisms of the Koszul homology of a local ring]{Automorphisms of the Koszul homology\\ of a local ring}
\author{Srikanth B.~Iyengar}
\author{Henrik Rüping}
\author{Marc Stephan}

\date{\today}

\address{Department of Mathematics, University of Utah, Salt Lake City, UT 84112, U.S.A.}
\email{iyengar@math.utah.edu}

\address{Continentale Krankenversicherung a.G., Ruhrallee 94, 44139 Dortmund, Germany}
\email{henrikrueping@googlemail.com}
\address{Faculty of Mathematics, Bielefeld University, PO Box 100 131, 33501 Bielefeld, Germany}
\email{marc.stephan@math.uni-bielefeld.de}

\begin{document}

\thanks{Partly supported by NSF grant DMS-200985 (SBI)}

\begin{abstract}
This work concerns the Koszul complex $K$ of a commutative noetherian local ring $R$, with its natural structure as differential graded $R$-algebra. It is proved that under diverse conditions, involving the multiplicative structure of $H(K)$, any dg $R$-algebra automorphism of $K$ induces the identity map on $H(K)$. In such cases, it is possible to define an action of the automorphism group of $R$ on $H(K)$. On the other hand, numerous rings are described for which $K$ has automorphisms that do not induce the identity on $H(K)$. For any $R$, it is shown that the group of automorphisms of $H(K)$ induced by automorphisms of $K$ is abelian.
\end{abstract}

\maketitle

\section{Introduction}
\label{sec:introduction}
This paper concerns automorphisms of the Koszul homology of a local ring induced by automorphisms of its Koszul complex. To set the stage for the discussion, we begin with a (commutative, noetherian) local ring $R$ and the Koszul complex $K$ on a minimal generating set for the maximal ideal of $R$, viewed as a differential graded $R$-algebra. The homology of $K$, denoted $H(K)$, has a structure of a graded-commutative $k$-algebra, where $k$ is the residue field of $R$. It is by know well understood that the structure of $K$ and of $H(K)$ capture interesting information about the ring $R$. So it is natural to study dg $R$-algebra automorphisms of $K$, and the automorphisms of $H(K)$ induced by them. But in fact we were led to study these because of recent work~\cite{ruepingstephan2019} of the second and third authors on transformation groups, answering a question of Walker and the first author~\cite{iyengarwalker2017}. In \cite{ruepingstephan2019} the problem arose of lifting automorphisms of the ring $R$ to those of $H(K)$; it was proved that this is possible when $R$ is the group algebra over $\FF_p$ of an elementary abelian $p$-group. As explained in \cref{rem:automorphisms}, the obstructions to lifting are precisely automorphisms of the dg $R$-algebra $K$ that are nontrivial on $H(K)$.

The present work started with the discovery that there are other interesting families of local rings $R$ with this property, but also families of local rings with nontrivial actions on $H(K)$. Here is a summary of our findings. The integer $c$ is the embedding codepth of $R$, namely $\edim R - \depth R$.

\begin{theorem}
\label{thm:listofThings} 
Let $R$ be a noetherian local ring and $K$ the Koszul complex on a minimal generating set for the maximal ideal of $R$.
Set $c=\sup\{ i \mid H_i(K)\ne 0\}$. Let $\varphi$ be an automorphism of the dg $R$-algebra $K$. The following statements hold.
\begin{enumerate}
\item\label{item:extremaldegrees}
$H_1(\varphi)$ and $H_c(\varphi)$ are the identity maps.
\item\label{item:ci}
$H(\varphi)$ is the identity when $R$ is a complete intersection. 
\item\label{item:gorenstein}
If $R$ is Gorenstein and $H_i(\varphi)$ is the identity, so is $H_{c-i}(\varphi)$.
\item\label{item:trivialproduct}
$H(\varphi)$ is the identity when $H_1(K)\cdot H_{i\geq 1}(K)=0$.
 \end{enumerate}
\end{theorem}
Part (\ref{item:extremaldegrees}) is contained in \cref{lem:H1} and \cref{lem:codepth}; part (\ref{item:ci}) is \cref{cor:ci}; for (\ref{item:gorenstein}) see \cref{lem:gorenstein}; and for (\ref{item:trivialproduct}) see \cref{cor:trivial-product}.

It follows from \cref{thm:listofThings}(\ref{item:extremaldegrees}) that when the embedding codepth of $R$ is at most $2$, the map $H(\varphi)$ is always the identity. The hypothesis in \cref{thm:listofThings}(\ref{item:trivialproduct}) holds when the ring $R$ is Golod, but not only; see the discussion following \cref{cor:trivial-product} and \cref{ex:destefani}. We provide numerous examples that suggest that the results above are the best possible. For instance, there are local rings of embedding codepth $3$ with nontrivial $H(\varphi)$; see \cref{ex:aci}. Also, it follows from \cref{thm:listofThings}(\ref{item:extremaldegrees}) and \cref{thm:listofThings}(\ref{item:gorenstein}) above that $H(\varphi)$ is always the identity when $R$ is a Gorenstein ring of embedding codepth $3$, but this does not extend to embedding codepth $4$; see \cref{ex:gorenstein}. 

On the other hand, for any $R$ and $\varphi$, the map $\varphi$ induces the identity on the associated graded of $H(K)$; see \cref{cor:graded}. These results seem to suggest that it is difficult to characterize rings with the property that every automorphism of the dg algebra $K$ induces the identity map on $H(K)$. 

To quantify the failure of this property, we consider the group, say $A$, of automorphisms of the $k$-algebra $H(K)$ induced by automorphisms of the dg $R$-algebra $K$. The automorphisms of $K$ themselves form a group under composition, and that is typically far from abelian. The result below, contained in \cref{thm:isgrouphomomorphism}, came as something of a surprise to us.

\begin{theorem}\label{thm:structureOfA}
The group $A$ of induced automorphisms on $H(K)$ is abelian. When the residue field of $R$ is of positive characteristic $p$, then $A$ is either trivial or has exponent $p$.
\end{theorem}

The exponent property is a feature of rings of positive characteristic, for in \cref{ex:smllWithQ} we describe a ring $R$ for which $A$ contains elements of infinite order.

Our goal in this work is to study actions on Koszul homology induced by automorphisms of the ring, or of the Koszul complex viewed as a dg algebra. The issues that arise will be more transparent when cast in a broader framework, which is what is done in \cref{sec:framework}. Many of our results are proved in this generality. 

\section{A framework}\label{sec:framework}
In what follows we use some basic notions and constructions pertaining to dg modules over dg algebras; \cite{avramov1998} is our reference for this material.
Throughout we write ``$(R,\fm,k)$ is a local ring" to mean that $R$ is a noetherian local ring, with maximal ideal $\fm$, and residue field $k$. A \emph{local homomorphism} $\alpha\colon R\to (S,\fn)$ is a map of local rings such that $\alpha(\fm)\subseteq \mathfrak{n}$. It induces a map on residue fields $k \to S/\fn$.

\begin{definition}
\label{de:categoryC}
Consider the category, denoted $\FreeCovers$, with objects the maps $f\colon F\to R$ where $(R,\fm,k)$ is a local ring, $F$ is a finite free $R$-module and $f$ is an $R$-linear map satisfying the conditions
\[
\im(f)=\fm \quad\text{and}\quad \ker(f)\subseteq \fm F\,.
\]
Given that $\im(f)=\fm$, the second condition is equivalent to the condition that the surjection $f\colon F\to \fm$ induces an isomorphism $F/\fm F \to \fm/\fm^2$; this is by Nakayama's Lemma. Thus $f$ is a \emph{free cover} of $\fm$. A morphism in $\FreeCovers$ is a pair 
\[
(\phi,\alpha)\colon (F\xra{f} R)\longrightarrow (G\xra{g}S)
\]
where $\alpha\colon R\to S$ is a local homomorphism, $\phi$ is a $\alpha$-equivariant map (i.e.~a map of abelian groups with $\phi(rx)=\alpha(r)\phi(x)$ for all $r\in R$, $x\in F$), and the following diagram commutes
\[
\begin{tikzcd}
F \arrow{d}[swap]{f} \arrow{r}{\phi} & G\arrow{d}{g}\\
R \arrow{r}[swap]{\alpha} & S
\end{tikzcd}.
\]
It will be convenient to speak of $\phi$ as a \emph{lift} of $\alpha$. Since $\alpha(\fm)\subseteq \mathfrak{n}$, lifts exist but there are typically many.

We are mainly interested in the automorphisms in $\FreeCovers$, namely, maps 
\[
(\phi,\alpha)\colon (F\to R) \longrightarrow (F\to R)
\]
where $\alpha$ is an automorphism of $R$; this implies that $\phi$ is an isomorphism. Then we can focus on the case $\alpha=\id_R$.
\end{definition}

\begin{definition}
\label{de:Koszul-functor}
Given a map $f\colon F\to R$ as above, we write $K(f)$ for its Koszul complex viewed as a dg $R$-algebra in the usual way; $K(f)$ is an exterior algebra on $F$ with differential induced by $f$; see \cite[\S1.6]{brunsherzog1998}. Thus, $K(f)$ is strictly graded-commutative (in that, the square of an element of odd degree is zero), and hence so is its homology algebra that we denote $H(f)$. The notation, adopted from \cite[Definition~1.6.3]{brunsherzog1998}, is potentially confusing, for it is not the homology of $f$, but the context should dispel any confusion. 

Any morphism $(\phi,\alpha)$ in $\FreeCovers$ extends to a $\alpha$-equivariant morphism of dg algebras
\[
K(\phi,\alpha)\colon K(f) \longrightarrow K(g)\,.
\]
This defines a functor $K$ from the category $\FreeCovers$ to the category of graded commutative dg algebras over local rings and morphisms given by pairs of a local ring homomorphism $\alpha$ and a $\alpha$-equivariant map of dgas. The functor $K$ is fully faithful: it is faithful since $K_0(\phi,\alpha)=\alpha$, $K_1(\phi,\alpha)=\phi$, and it is full since $K(f)$ is an exterior algebra on $K_1(f)=F$.

The morphism $K(\phi,\alpha)$ induces a map of graded rings
\[
H(\phi,\alpha)\colon H(f) \longrightarrow H(g)\,.
\]
Observe that $H(f)$ is an algebra over the residue field of $R$, whereas $H(g)$ is an algebra over the residue field of $S$. The map above is equivariant with respect to the map on the residue fields induced by $\alpha$.
\end{definition}

The case $\alpha=\id_R$ provides the setting in the introduction. Since $K$ is fully faithful, every dga-endomorphism $\varphi$ of $K(f)$ over $R$ is determined by the lift $\varphi_1\colon K_1(f)\to K_1(f)$ of $\id_R$; and $\varphi$ is automatically an automorphism since $\id_R$ is an isomorphism.

\begin{remark}
\label{rem:automorphisms}
Let $(R,\fm,k)$ be a local ring and $f\colon F\to R$ a free cover of $\fm$. By functoriality, the automorphism group $\aut_{\FreeCovers}(f)$ acts on the Koszul homology $H(f)$ via graded $\ZZ$-algebra automorphisms:
\[\aut_{\FreeCovers}(f)\longrightarrow \aut(H(f)), \quad (\phi,\alpha)\mapsto H(\phi, \alpha).
\]
In the next section and beyond we encounter various contexts where $H(\phi,\alpha)$ is independent of the lift $\phi$. The result below is the main reason for our interest in such independence. In the statement, $\aut(R)$ denotes the group of automorphisms of the ring $R$. We wish to define a compatible $\aut(R)$ action on the Koszul homology, $H_i(f)$, for a given $i$. Compatibility means that the following diagram commutes:
\[
\begin{tikzcd}
\aut_{\FreeCovers}(f)\arrow{r}\arrow[twoheadrightarrow]{d}&\aut_\ZZ(H_i(f))\\
\aut(R)\arrow[dashrightarrow]{ru}
\end{tikzcd}
\]
The surjective map sends an element $(\phi,\alpha)$ to $\alpha$. 

\begin{proposition}
\label{pro:automorphisms}
Let $f\colon F\to R$ be a free cover of $\fm$. If for some integer $i$ each automorphism of the dg $R$-algebra $K(f)$ induces the identity map on $H_i(f)$, then the $\aut_\FreeCovers(f)$-action on $H_i(f)$ descends to an $\aut(R)$-action.
\end{proposition}

\begin{proof}
 We need to show that every element $(\phi,\alpha)$ in the kernel of the surjection $\aut_{\FreeCovers}(f)\to\aut(R)$ induces the identity on $H_i(f)$. This holds since the kernel consists of the pairs $(\phi,\id_R)$ and $K(\phi,\id_R)$ is an automorphism of the dg $R$-algebra $K(f)$ so that $H_i(\phi,\alpha)$ is the identity by assumption.
\end{proof}

 If the $\aut_\FreeCovers(f)$-action on $H_i(f)$ descends to an $\aut(R)$-action, then the subgroup of ring automorphisms that induce the identity on $k=R/\fm$ acts via automorphisms of $k$-vector spaces.
\end{remark}

Next we recall an identification of cycles in degree one in Koszul complexes; see \cite[Section~2.3]{brunsherzog1998} for details. In what follows we typically describe $f$ by specifying a matrix; it is assumed that a basis for $F$ has been chosen, and this basis is denoted $e_1,\ldots,e_n$.

\begin{remark}
\label{rem:degree-one-cycles}
Suppose that $\pi\colon Q\to R$ is a minimal Cohen presentation; thus, $(Q,\fq,k)$ is a regular local ring, and $\pi$ is surjective with $\ker(\pi)\subseteq \fq^2$. Such a presentation exists when $R$ is complete with respect to the topology induced by its maximal ideal. Fix minimal generating sets $x_1,\dots,x_n$, and $q_1,\dots,q_c$, for the ideals $\fq$ and $\ker(\pi)$, respectively. The minimality of the Cohen presentation means that there exist elements $\{a_{ij}\}$ in $\fq$, for $1\le i\le c$ and $1\le j\le n$ such that
\[
q_i = \sum_j a_{ij}x_j \quad \text{for each $i$.} 
\]
Let $\overline{x}_1,\dots,\overline{x}_n$ be the images in $R$ of the elements $x_1,\dots,x_n$; they are a minimal generating set for the ideal $\fm$, again by the minimality of $\pi$. Set 
\[
f\colon F\xra{ \begin{pmatrix} \overline{x}_1 & \dots& \overline{x}_n \end{pmatrix} } R\,.
\]
Then $f$ is an object in $\FreeCovers$. Set
\[
z_i \coloneqq \sum_j \overline{a}_{ij}e_j \quad \text{for $1\le i\le c$.} 
\]
It is readily checked that these elements are in $\ker(f)$. Hence they are cycles in $K_1(f)$. In fact they form a basis for the homology in degree one:
\[
H_1(f) = \bigoplus_i k[z_i]\,.
\]
Indeed, it is easy to check that classes $\{[z_i]\}_{i=1}^c$ are linearly independent. It remains to observe that the $k$-vector space has rank $c$. This fact can be deduced from the isomorphism of graded $k$-vector spaces
\[
H(f) \cong \Tor^Q(R,k)\,.
\]
In degree one the isomorphism above reads
\[
H_1(f) \cong \frac{I}{\fq I} \qquad\text{where $I=\ker(\pi)$.}
\]
Thus $\rank_k H_1(f)$ is the minimal number of generators of the ideal $I$, that is to say, $c$. One thinks of the vector space $I/\fq I$ as the space of relations defining $R$.
\end{remark}

\begin{remark}\label{rem:examplespolynomialquotients}
In the following sections we present numerous examples to illustrate the scope and limitations of our results. The rings $R$ that appear will typically be quotients of standard graded polynomial rings, by a homogeneous ideal. This makes it possible to use computer algebra systems (our choice is Macaulay 2~\cite{M2}) for calculations; see \cref{rem:M2}. When $R$ is standard graded, the only homogeneous automorphism of $F\to R$ that lifts the identity on $R$ is the identity map on $F$. Therefore, the lifts we consider will not be homogeneous. One can get examples of local rings from the graded ones, by localising at the homogeneous maximal ideal, or completing at that ideal.

In detail, say $k$ is a field and $P\coloneqq k[x_1,\dots,x_n]$ the polynomial ring in indeterminates $\boldsymbol{x} \coloneqq x_1,\dots,x_n$, viewed as a standard graded ring, that is to say, with each $x_i$ of degree one. Let $R=P/I$ where $I$ is homogeneous ideal in $(\boldsymbol{x})^2$, and let $\fm$ be the maximal ideal $(\boldsymbol{x})R$ of $R$. When $\dim R=0$, equivalently, when $\fm$ is nilpotent, the ring $R$ is local already. Else we can pass to the local ring $R_\fm$, or its completion at $\fm$. Our results involve the structure of the Koszul homology ring, and the Koszul complex of $R$ on the sequence $\boldsymbol{x}$ is quasi-isomorphic to the Koszul complex of $R_\fm$, and also to the Koszul complex of its completion, even as dg algebras.

The Koszul complex $K$ of $R$ on the sequence $\boldsymbol{x}$ acquires an internal grading coming from the grading on $R$, so its homology is bigraded. It is helpful to display the ranks of this bigraded object in a table whose entry in the $i$th column and $j$th row is the rank of $H_i(K)_{i+j}$, the component in degree $(i+j)$ of the $i$th Koszul homology; see \cref{ex:smllWithQ}. This is the \emph{Betti table} of $R$, and it records the ranks of the free modules (with twists) in the minimal free resolution of $R$ over $P$; see \cite[\S1B]{eisenbud2005}. The index of the last nonzero column in the Betti table is the projective dimension of $R$ over $P$, whereas the index of the last nonzero row is its regularity. 
\end{remark}

\section{Extremal degrees}
\label{sec:extremaldegrees}
Throughout this section, we fix a morphism
\[
(\phi,\alpha)\colon (F\xra{f} R)\longrightarrow (G\xra{g}S)
\]
in $\FreeCovers$. The issue that concerns us is the dependence of the map $H(\phi,\alpha)$ on $\phi$. We will show that independence always holds in degree one and if $\alpha$ is an isomorphism, in the largest degree $i$ with $H_i(f)\neq 0$.

\begin{remark}\label{rem:reduceisotoidentity}
Let $\alpha\colon R\to S$ be an isomorphism. The map $H_i(\phi,\alpha)$ is independent of the lift $\phi$ of $\alpha$ if and only if any lift of $\id_R$ induces the identity on $H_i(f)$.

Indeed, suppose $H_i(\phi,\alpha)$ is independent of $\phi$. Given any lift $\psi$ of $\id_R$ the map $\phi\circ\psi$ is also a lift of $\alpha$, so the hypothesis gives the second equality below:
\[
H_i(\phi,\alpha)\circ H_i(\psi,\id_R) = H_i(\phi\circ\psi,\alpha) = H_i(\phi,\alpha)\,.
\]
Since $H_i(\phi,\alpha)$ is an isomorphism, $H_i(\psi,\id_R)$ is the identity map, as claimed.

Conversely, say $\phi$ and $\phi'$ are lifts of $\alpha$. If each lift of $\id_R$ induces the identity on $H_i(f)$, then since $\phi^{-1}\circ \phi'$ is a lift of $\id_R$, one has
\[
\id=H_i((\phi^{-1},\alpha^{-1}) \circ(\phi',\alpha)) =H_i(\phi,\alpha)^{-1}\circ H_i(\phi',\alpha)
\]
 and hence $H_i(\phi,\alpha)=H_i(\phi',\alpha)$.
\end{remark}

\begin{lemma}
\label{lem:H1}
If $\phi$ and $\psi$ are lifts of a local homomorphism $\alpha$, then 
\[
H_1(\phi,\alpha)=H_1(\psi,\alpha)\,.
\]
Consequently $H(\phi,\alpha)=H(\psi,\alpha)$ on the subalgebra of $H(f)$ generated by $H_1(f)$.
\end{lemma}

\begin{proof}
For any element $x\in F$, the difference
\[\phi(x)-\psi(x)
\]
lies in $\ker g$, i.e., it is a cycle in $K_1(g)$.

Since $\ker f\subseteq \fm F$, a cycle $z\in K_1(f)$ is of the form 
\[
z=\sum_i r_ix_i \quad\text{with $r_i\in \fm$ and $x_i\in K_1(f)$}.
\]
Therefore we get that
\[
\phi(z) - \psi(z) = \sum_i \alpha(r_i) (\phi(x_i)-\psi(x_i))\,.
\]

We have already observed that the $\phi(x_i)-\psi(x_i)$ are cycles. Since $\alpha(r_i)$ is in $\mathfrak{n}$ and $\mathfrak{n} H(g)=0$ (see \cite[Proposition~1.6.5 (b)]{brunsherzog1998}) it follows that each $\alpha(r_i) (\phi(x_i)-\psi(x_i))$ is a boundary and hence so is their sum. We conclude that $\phi(z)-\psi(z)$ is a boundary in $K(g)$, which is the desired result.
\end{proof}

\begin{remark}
\label{rem:cirings}
Fix an object $f\colon F\to R$ in $\FreeCovers$. Since $H(f)$ is graded-commutative, there is a map of $k$-algebras
\[
\bigwedge_k \Sigma H_1(f) \longrightarrow H(f)
\]
This map is an isomorphism if, and only if, the ring $R$ is complete intersection; see \cite[Theorem~2.3.11]{brunsherzog1998}.
\end{remark}

The preceding remark and \cref{lem:H1} have the following consequence. The result seems optimal for it does not extend to Gorenstein rings or to almost complete intersections; see \cref{ex:aci} and \cref{ex:gorenstein}.

\begin{corollary} 
\label{cor:ci}
Let $(\phi,\alpha)$ be a morphism in $\FreeCovers$, with source $f\colon F\to R$. When $R$ is complete intersection the map $H(\phi,\alpha)$ depends only on $\alpha$. In particular, any $R$-algebra automorphism $\varphi\colon K(f)\to K(f)$ induces the identity on $H(f)$.
\end{corollary}

\begin{proof}
Since $R$ is a complete intersection, $H(f)$ is an exterior algebra in $H_1(f)$; see \cref{rem:cirings}. Thus the subalgebra generated by $H_1(f)$ is the whole of $H(f)$. 
\end{proof}

With regards to the preceding result, it should be stressed that although for a fixed $\alpha\colon R\to S$, the various maps $(\phi,\alpha)$ induce the same map on homology, they need not be homotopic, even when $\alpha$ is the identity; see \cref{ex:not-hequivalent}.

\begin{example}
Let $k$ be a field and $R\coloneqq \pos{x_1,\dots,x_n}/(x_1^{d_1},\dots, x_c^{d_c})$, where the $d_i\ge 2$ are integers. Let $f\colon F\to R$ be the map from \cref{rem:degree-one-cycles}. Then the cycles
\[
[x_1^{d_1-1}e_1], \cdots, [x_c^{d_c-1}e_c]
\]
are a basis for the $k$-vector space $H_1(f)$.

Let $\alpha\colon R\to R$ be a morphism of $k$-algebras, defined by
\[
\alpha(x_i) = \sum_{j}r_{ij} x_j \qquad \text{for $i=1,\dots,n$.}
\]

For any lift $\phi$ of $f\colon F\to R$, the endomorphism $H(\phi,\alpha)$ of $H(f)$ is given by
\[
[x_i^{d_i-1}e_i] \mapsto [\sum_{j} \alpha(x_i)^{d_i-1}r_{ij}e_j] 
\]
By \cref{cor:ci}, this map is independent of the choice of the lift.
\end{example}

The next example is a special case of the last one, but there is a new feature.

\begin{example}
Let $p$ be a prime number and $k=\FF_p$, the field with $p$ elements. In the polynomial ring $k[x_1,\dots,x_n]$ consider the ideals 
\[
\fq \coloneqq (x_1,\dots,x_n) \qquad\text{and}\qquad I \coloneqq (x_1^p,\dots,x_n^p)\,.
\]
Set $R\coloneqq k[x_1,\dots,x_n]/I$; the group algebra of an elementary abelian $p$-group of rank $n$ has this form. The ring $R$ is local with maximal ideal $\fm \coloneqq \fq/I$. Let $f\colon F\to R$ be a free cover of $\fm$; in particular, $H_1(f) \cong I/\fq I$, as explained in \cref{rem:degree-one-cycles}.

One can relate $H_1(f)$ to the space of generators $\fq/\fq^2$ of $R$, as follows: The Frobenius endomorphism $\varphi$ of $k[x_1,\dots,x_n]$ induces a $k$-linear isomorphism between the space of generators and the space of relations of $R$:
\[
\frac{\fq}{\fq^2} \xra[\cong]{\ \varphi\ } \frac{I}{\fq I}\,.
\]

Let now $\alpha\colon R\to R$ be a $k$-algebra map. This can be lifted to a $k$-algebra map $\wt\alpha$ of $ k[x_1,\dots,x_n]$ such that $\wt\alpha(I)\subseteq I$. With $\phi$ any lift of $\alpha$, one then gets a commutative diagram
\[
\begin{tikzcd}
\frac{\fq}{\fq^2} \arrow{d}[swap]{\alpha} \arrow{r}{\varphi}[swap]{\cong} 
  & \frac{I}{\fq I} \arrow{d}{\wt\alpha} \arrow{r}[swap]{\cong} & H_1(f) \arrow{d}{H_1(\phi,\alpha)} \\
\frac{\fq}{\fq^2} \arrow{r}{\varphi}[swap]{\cong} & \frac{I}{\fq I} \arrow{r}[swap]{\cong} & H_1(f) \\
\end{tikzcd}.
\]
This diagram reconciles the description of the action of $\alpha$ on $H_1(f)$ with the one given in \cite[Proposition~8.8]{ruepingstephan2019}.
\end{example}

See the paragraph following \cref{cor:ci} for the import of the example below. When $k=\FF_2$, the ring $R$ in question is the group algebra, over $k$, of the elementary abelian $2$-group of rank two.

\begin{example}
\label{ex:not-hequivalent}
Let $k$ be a field and set 
\[
R\coloneqq \frac{k[x_1,x_2]}{(x_1^2,x_2^2)}.
\]
Let $F$ be the free $R$-module on basis $e_1,e_2$ and $f\colon F\xra{(x_1 \; x_2)}R$ a free cover of the maximal ideal $(x_1,x_2)$ of $R$. Consider the morphism of dg $R$-algebras
\[
\varphi\colon K(f) \longrightarrow K(f)\,, \quad e_1\mapsto e_1+x_2 e_2, \quad e_2\mapsto e_2.
\]
Then $H(\varphi)$ is the identity, by \cref{cor:ci}; this can be checked directly. However $\varphi$ is not chain homotopic to the identity, as morphisms of $R$-complexes. For example, if it were, so would the induced map 
\[
S\otimes_R\varphi \colon S\otimes_R K(f)\longrightarrow S\otimes_R K(f) \quad \text{where $S\coloneqq R/x_1R$.}
\]
However the $k$-vector space $H_1(S\otimes_R K(f))$ has a basis $[e_1],[x_2e_2]$, and the map induced by $S\otimes_R\varphi$ sends $[e_1]$ to $[e_1] + [x_2e_2]$.
\end{example}

The ring $R$ in the next example complements \cref{cor:ci}, for it is an almost complete intersection for which there are automorphisms of the Koszul complex inducing non-identity maps on the homology.

\begin{example}
\label{ex:aci}
Let $R$ be the ring $\pos[\FF_2]{t^6,t^{10},t^{14},t^{15}}\subset \pos[\FF_2]{t}$ and consider the free cover of its maximal ideal $
f\colon F\to R$ represented by the matrix $\begin{pmatrix} t^6 & t^{10} & t^{14} & t^{15} \end{pmatrix}$.

The Koszul complex of $f$ admits an automorphism that does not induce the identity on $H_2(f)$.
It fixes $e_i$ for $i\ge 2$ and sends $e_1$ to $e_1+t^{16}e_3+t^{15}e_4$.
Thus the class
\[
[(t^{18}e_2+t^{14}e_3)e_1+t^{10}e_2e_3]
\]
is sent to itself plus $[t^{18}e_2+t^{14}e_3]\cdot [t^{16}e_3+t^{15}e_4]$. We checked by hand and verified with Macaulay2 that this product is not zero in $H_2(f)$.

Since $\dim R=1$ the highest nonzero Koszul homology module is $H_3(f)$, and on this module any automorphism induces the identity map; see \cref{lem:codepth}.
\end{example}

Next we examine how different lifts of $\alpha$ can be chosen.

\begin{remark}
\label{rem:actiononlifts}
It is easy to check that if $\delta\colon F\to \ker(g)$ is a $\alpha$-equivariant map, then $(\phi+\delta,\alpha)$ is also a morphism in $\FreeCovers$. This defines an action of the abelian group $\hom_\alpha(F,\ker(g))$ of $\alpha$-equivariant maps on the set of lifts of $\alpha$. This group action is free. It is also transitive, since the difference of two lifts lands in $\ker(g)$.
\end{remark}

\begin{lemma}\label{lem:embeddingDim}
Let $n=\rank_R F$. If $\alpha\colon R\to S$ is a local homomorphism such that $\alpha(\fm)=\mathfrak{n}$, then $H_n(\phi,\alpha)$ is independent of the chosen lift $\phi$.
\end{lemma}

\begin{proof} Let $e_1,\ldots, e_n$ be a basis of $F$. Thus $K_n(f)$ is a free $R$-module of rank one, generated by $e_1\wedge\ldots\wedge e_n$. It follows from the construction of the Koszul complex that $r e_1\wedge\ldots\wedge e_n$ is a cycle, if and only if $r$ is in $\Ann(\fm)$, the annihilator of the maximal ideal $\fm$; this is the socle of $R$. Since $K_{n+1}(f)=0$, we get that $H_n(f)$ is $\Ann(\fm)(e_1\wedge \ldots \wedge e_n)$.

We want to show that $H_n(\phi+\delta,\alpha)=H_n(\phi,\alpha)$ for any $\delta$ in $\hom_\alpha(F,\ker(g))$.

By assumption, $\alpha(\fm)= \mathfrak{n}$ and thus $\alpha(\Ann(\fm))\subseteq \Ann(\mathfrak{n})$. Since the element $\delta(e_i)$ is in $\ker(g)\subseteq \mathfrak{n}G$, one has $\alpha(r)\delta(e_i)=0$. This gives the third equality below.
\begin{align*}
H_n(\phi+\delta,\alpha)([re_1\wedge \ldots \wedge e_n])
&=[\alpha(r) (\phi+\delta)(e_1) \wedge \ldots \wedge (\phi+\delta)(e_n)]\\
&= [\alpha(r) (\phi(e_1)+ \delta(e_1)) \wedge \ldots \wedge (\phi(e_n)+\delta(e_n))]\\
&=[\alpha(r)\phi(e_1)\wedge \ldots \wedge\phi(e_n)]\\
&=H_n(\phi,\alpha)([re_1\wedge \ldots \wedge e_n])
\end{align*}
The others are by the definition of the maps involved. 
\end{proof}

\begin{remark}
\label{rem:depth}
Let $R$ be a local ring and $f\colon F\to R$ a free cover of its maximal ideal. We write $\edim R$ for the \emph{embedding dimension} of $R$, which is the minimal number of generators of the maximal ideal; equivalently, the rank of $F$. 

The \emph{depth} of a finitely generated $R$-module $M$, denoted $\depth_RM$, is the length of the maximal $M$-regular sequence in $R$; see \cite[Definition~1.2.7]{brunsherzog1998}. The depth sensitivity of Koszul complexes~\cite[Theorem~1.6.17]{brunsherzog1998} yields that 
\[
\sup\{i\ge 0\mid H_i(K(f)\otimes_RM)\ne 0\} = \edim R - \depth_RM\,.
\] 
Thus \cref{lem:embeddingDim} is only interesting if $\depth R=0$ and $\depth S \le \edim S-\edim R$. A better statement holds in the special case when $\alpha$ is the identity map.
\end{remark}

\begin{lemma}
\label{lem:codepth}
Let $\alpha\colon R\to S$ be an isomorphism of local rings. For $c=\edim R -\depth R$ one has $H_c(\phi,\alpha)$ is independent of the lift $\phi$. 
\end{lemma}

\begin{proof}
It suffices to treat the case $\alpha=\id_R$; see \cref{rem:reduceisotoidentity}. It helps to consider a more general statement: For any finitely generated $R$-module $M$, set 
\[
H_i(f,M)\coloneqq H_i(K(f)\otimes_RM)\qquad\text{and}\qquad 
  H_i(\phi,M) = H_i(K(\phi,\id_R)\otimes_RM)\,.
\]
This is a self-map of $H_i(K(f)\otimes_RM)$. The desired result is the case $M=R$ of the following claim: $H_c(\phi,M)$ is independent of the lift $\phi$, for $c=\edim R-\depth_RM$.

We verify this claim by induction on $\depth_RM$. The base case is $\depth_RM=0$, and then the claim follows by an argument analogous to the one for \cref{lem:embeddingDim}.

Assume $\depth_RM\ge 1$, let $x$ be an $M$-regular element, and consider the exact sequence of $R$-modules
\[
0\rightarrow M\xra{x} M \rightarrow \overline{M}\rightarrow 0\,.
\]
Since the maximal ideal of $R$ annihilates $H_i(f,M)$ one gets that $x\cdot H_i(f,M)=0$. Thus the exact sequence above induces exact sequences
\[
0\longrightarrow H_i(f,M)\longrightarrow H_i(f,\overline{M})\longrightarrow H_{i-1}(f,M)\longrightarrow 0.
\]
Moreover these are compatible with the maps $H_i(\phi,M)$ and $H_i(\phi,\overline{M})$. Since $H_{c+1}(f,M)=0$, the exact sequence above gives an isomorphism
\[
H_{c+1}(f,\overline{M}) \xrightarrow{\ \cong\ } H_c(f,M)\,.
\]
As $\depth_R\overline{M}=\depth_RM-1$, the induction hypothesis implies that the map $H_{c+1}(\phi,\overline{M})$ is independent of the chosen lift $\phi$, so the same is true of $H_c(\phi,M)$.
\end{proof}

\section{Products in Koszul homology}
\label{sec:prodKoszulHomology}

Throughout this section, we fix a morphism
\[
(\phi,\alpha)\colon (F\xra{f} R)\longrightarrow (G\xra{g}S)
\]
in $\FreeCovers$. In the previous section we found conditions on the source that guarantee that $H(\phi,\alpha)$ is independent of $\phi$. In this section we will establish independence of $H(\phi,\alpha)$ if $H_1(g)\cdot H_{\ge 1}(g)=0$. Moreover, we will show that if $R$ is Gorenstein and $\alpha$ an isomorphism, then the independence of $H_i(\phi,\alpha)$ is equivalent to the independence of $H_{c-i}(\phi,\alpha)$, where $c$ denotes the codepth of $R$.

\begin{lemma} \label{lem:conntomult}
For any $\delta$ in $\hom_{\alpha}(F,\ker(g))$ and $z$ in $H_m(f)$, the difference
\[
H_m(\phi+\delta,\alpha)(z)- H_m(\phi,\alpha)(z)
\]
lies in $H_1(g)\cdot H_{m-1}(g)$.
\end{lemma}

\begin{proof}
Let $e_1,\ldots,e_n$ be a basis of $F$. We can assume without loss of generality that $\delta(e_j)=0$ for all $j\neq i$, where $1\leq i\leq n$ is fixed.

 Let $L\subseteq K(f)$ be the graded subalgebra spanned by the $e_j$'s for $j\neq i$. Note that $L$ is closed under the differential of $K(f)$. Since the underlying algebra of $K(f)$ is an exterior algebra, the graded sub-modules $L\subseteq K(f)$ and $e_iL\subseteq K(f)$ are complementary:
\[ 
e_i\cdot L\oplus L = K(f).
\]
Fix $z\in H_m(f)$. Choose a cycle representing $z$ and write it in the form
\[
e_i\cdot r+r',
\]
where $r$ and $r'$ lie in $L$. Applying the differential to this cycle, we get
\[0 = d(e_i\cdot r+r')= d(e_i)\cdot r-e_i\cdot d(r)+d(r').\]
The first and the third summand lie in $L$, and the second term lies in $e_i\cdot L$. Thus we can conclude that $d(r)=0$.

Since $\delta$ fixes each element of $L$, the maps $K(\phi+\delta,\alpha)$ and $K(\phi,\alpha)$ agree on $r$ and on $r'$. It follows that
\[(K(\phi+\delta,\alpha)-K(\phi,\alpha))(e_i\cdot r+r')= \delta(e_i)\cdot K(\phi,\alpha)(r).\]
Since both factors on the right-hand side are cycles, we get that 
\[H(\phi+\delta,\alpha)(z)-H(\phi,\alpha)(z)=[\delta(e_i)]\cdot H(\phi,\alpha)([r])\]
is a product in homology as well.
\end{proof}

\begin{remark}
Let $e_1,\ldots, e_n$ be a basis of the free $R$-module $F$ and fix $1\leq i\leq n$. Our proof establishes the decomposition
\[H_m(\phi+\delta,\alpha)(z)- H_m(\phi,\alpha)(z) \in [\delta(e_i)]\cdot H_{m-1}(g)
\]
if $\delta\in \hom_\alpha(F,\ker(g))$ satisfies $\delta(e_j)=0$ for all $j\neq i$.
\end{remark}

Observe that in the statement below the hypothesis on $H_{m-1}(g)$ depends only on $S$ and not on the chosen free cover $g$. Contrast this with \cref{cor:ci} which identifies conditions on $R$ that lead to a similar conclusion.

\cref{lem:conntomult} gives us families of local rings, outside complete intersections for which automorphism of their Koszul dg algebras induces the identity on homology. 

\begin{corollary}
\label{cor:trivial-product}
If $H_1(g) H_{m-1}(g)=0$ for some integer $m\ge 2$, then for any morphism $(\phi,\alpha)$ from 
$f\colon F\to R$ to $g\colon G\to S$ in $\FreeCovers$ the induced map
\[
H_m(\phi,\alpha)\colon H_m(f)\rightarrow H_m(g)
\]
is independent of the chosen lift $\phi$ of $\alpha$. \qed
\end{corollary}

The preceding result leads naturally to a search for local rings whose Koszul homology has trivial products. Golod rings fit this bill. These are local rings characterised by the property that their Koszul complex is quasi-isomorphic, as a dg algebra, to the trivial extension algebra $k\ltimes V$, where $V$ is a graded $k$-vector space; in particular, the product of any two elements in positive degree is zero; see~\cite{golod1962}, and also \cite[(2.3)~Theorem]{avramov1986}. These sources also describe various families of Golod rings; see also \cite[\S5.2]{avramov1998}.

\begin{example}
Let $Q$ be a regular local ring with maximal ideal $\fq$.

For any integer $n\ge 2$, the local ring $S\coloneqq Q/\fq^n$ is Golod. Thus, for example, over any field $k$ the following local ring is Golod:
\[
\frac{k[x,y]}{(x^2,xy,y^2)}
\]
Any quotient ring $S$ of $Q$ that is Cohen-Macaulay with $\dim S \ge \dim Q - 2$, but not Gorenstein, is Golod. Here is one such ring:
\[
\frac{\pos{x,y,z}}{(y^2-xz,z^2-x^2y, x^3-yz)}
\]
This ring is isomorphic to $\pos{t^3,t^4,t^5}\subseteq \pos{t}$, and so a one dimensional domain. In particular, it is Cohen-Macaulay. That it is not Gorenstein can be checked by applying \cite[Theorem~3.2.10]{brunsherzog1998} . In particular, the ring is Golod.

Rings of the form $Q/IJ$ where $I$ and $J$ are ideals in $\fq$ tend to be Golod. For example, the local ring 
\[
\frac{\pos{x,y,u,v}}{(x^2,y^2)(u^2,v^2)}
\]
is Golod. However, not every such quotient need to Golod; see \cref{ex:destefani} below. 
\end{example}

There are non-Golod local rings whose Koszul homology has trivial product structure; see, for example, Katth\"an~\cite[Theorem~3.1]{katthan2017}. The preceding corollary applies also to such rings.

The example below from \cite[Example~2.1]{destefani2016} illustrates that $H_1\cdot H_{\geq 1}$ can be trivial, though there are nontrivial products in Koszul homology.

\begin{example}
\label{ex:destefani}
Let $k$ be a field and set
\[
S\coloneqq \frac{k[x,y,z,w]}{(x,y,z,w)\cdot (x^2,y^2,z^2,w^2)}
\]
This is a local ring of Krull dimension $0$. It is not Golod because the Koszul algebra $H$ has $H_2 \cdot H_2 \neq 0$. Nevertheless $H_1 \cdot H_i=0$ for $i\geq 1$, so \cref{cor:trivial-product} applies.
\end{example}

Next we present some examples that illustrate that triviality of multiplication by $H_1$ is sufficient, but not necessary, for trivial action on Koszul homology.

\begin{example}
The following examples all concern polynomial rings over $\FF_2$, in the obvious variables. We specify only the defining relations of the ring.

\begin{table}[!ht]
  \centering  
 \begin{tabular}{c|p{5.5cm}}
defining ideal & multiplication on Koszul homology\\ \hline
$(x^2, y^2, xz, yz, z^2-xy)$ & $H_1 H_1=0$, but $H_1 H_2\neq 0$ \\ \hline
$(x^2,xy,y^2,xu,yv,xv-yu,u^2,uv,v^2)$ & $H_1^2=0 = H_1H_2$, but $H_1 H_3\neq 0$\\ \hline 
$(x^2, y^2, z^2, w^2, yz-xw)$ & $H_1 H_1\neq 0$ but $H_1 H_i=0$ for $i\geq 2$ 
 \end{tabular} 
\end{table} 
The quotient rings are all local, and of Krull dimension $0$. The second example is \cite[Example~6.5]{gheibi2019quasi}. All these rings have the property that every dga-automorphism of their Koszul complexes induces the identity on homology. For the first two, this holds by \cref{lem:H1}, \cref{lem:embeddingDim} and \cref{cor:trivial-product}. The last ring is \cref{ex:filtrationargument}.
\end{example}

\begin{remark}
\label{rem:M2}
Let $R$ be a quotient of a polynomial ring by a homogeneous ideal and $K$ the Koszul complex of $R$ with respect to the maximal ideal that is generated by the indeterminates as in \cref{rem:examplespolynomialquotients}. Given a lift $\phi\colon R^n\to R^n$ of $\id_R$ we can check if the induced map induces the identity on homology using the package {\tt DGAlgebras} \cite{mooremacaulay2} in Macaulay2. The command 
\[{\tt K = koszulComplexDGA(R,Variable=>"e")}\]
builds the Koszul complex of $R$ as a dga and denotes the generators of the underlying exterior algebra by $e_1,\ldots, e_n$. The induced map $K(\phi,\id_R)$ is created with the command \[{\tt Kphi = dgAlgebraMap(K,K, matrix\{\{}\phi(e_1),\ldots,\phi(e_n){\tt \}\})},\] by replacing $\phi(e_i)$ by the actual linear combination $ \lambda_{1,i}e_1+\ldots +\lambda_{n,i} e_n$
where $(\lambda_{i,j})\in R^{n\times n}$ is the matrix representing $\phi$.

A quick way to create the identity map on $K(f)$ is 
\[{\tt Kid=dgAlgebraMap(K,K,getBasis(1,K))},\] as the command {\tt getBasis(1,K)} writes the generators $e_1,\ldots,e_n$ as a matrix.

Finally, the command
\[
{\tt HH(toComplexMap~ Kphi)==HH(toComplexMap~Kid)}
\]
returns true if the two maps induce the same map on homology, and false otherwise.

When $k=R/\fm$ is finite one only needs to check $n\cdot \dim_k H_1(f)$ lifts of $\id_R$ to decide whether any lift induces the identity on homology; see \cref{rem:reducecasesforprogram}.
\end{remark}

None of our criteria work for the following example. We could only verify it by a computer program as explained in \cref{rem:M2}. 

\begin{example} 
According to Macaulay2 the ring 
\[
\frac{\FF_2[x,y,z,w]}{(x^2, y^2, z^2, w^3, yz-xw)}
\]
has the property that every automorphism of its Koszul dga induces the identity on homology. This ring is neither a complete intersection, nor is the multiplication on Koszul homology trivial. In fact, $H_1\cdot H_i = 0$ for $i\geq 2$ but $H_1\cdot H_1\neq 0$. 
\end{example}

Adding one relation as follows negates this property.

\begin{example} Consider the ring
\[
R\coloneqq \frac{\FF_2[x,y,z,w]}{(x^2, y^2, z^2, w^3, yz-xw, y w^2)}
\]
and the free cover $f\colon R^4 \xrightarrow{(x\; y\; z\; w\;)} R$. The map 
\[
e_1 \mapsto e_1+ z e_3\,,\qquad\text{and}\qquad e_i\mapsto e_i\quad\text{for $i\neq 1$,}
\]
does not induce the identity on $H_2(f)$. It sends the class $[w^2e_1e_2+ywe_2e_3]$ to
\[
[w^2e_1e_2+ywe_2e_3]+[w^2e_2]\cdot[ze_3]
\]
We used the computer algebra system Macaulay2 to verify that $0\neq [w^2e_2]\cdot [ze_3]$.
\end{example}

\begin{remark}\label{rem:gorenstein}
Let $R$ be a Gorenstein local ring and $K$ the Koszul complex on a free cover of $\fm$. Set 
\[
c=\sup\{i\mid H_i(K) \ne 0\}\,.
\]
Thus $c=\edim R - \dim R$, as discussed in \cref{rem:depth}, keeping in mind that $\depth R=\dim R$ since $R$ is Gorenstein. The product on $H(K)$ induces for each integer $1\le i\le c-1$ a pairing
\[
H_i(K) \times H_{c-i}(K) \longrightarrow H_c(K)\,.
\]
The Gorenstein condition is exactly that $H_c(K)\cong k$ and the pairing above is perfect for each $i$ in the given range; said otherwise, $H(K)$ is a Poincar\'e duality algebra. The import of the next result, which builds on \cref{lem:codepth}, is that this duality is reflected in the action of the automorphisms of $R$ on Koszul homology. 
\end{remark}

\begin{lemma}\label{lem:gorenstein}
Let $R$ be a Gorenstein local ring, $\alpha\colon R\to S$ an isomorphism of rings, and set $c=\edim R -\dim R$.
If for some integer $i$ the map $H_i(\phi,\alpha)$ is independent of the lift $\phi$, then so is the map $H_{c-i}(\phi,\alpha)$. 
\end{lemma}

\begin{proof} As before we can assume $\alpha=\id_R$; see \cref{rem:reduceisotoidentity}. Recall that $H(\phi,\id_R)$ and $H(\psi,\id_R)$ agree on $H_c(f)$ by \cref{lem:codepth}. The desired statement follows from a more general statement about a graded algebra $H=\bigoplus_{i=0}^c H_i$ that satisfies Poincar\'e duality: Given automorphisms $\phi, \psi\colon H\to H$ with $\phi_i =\psi_i$ and $\phi_c=\psi_c$, one has $\phi_{c-i} = \psi_{c-i}$.

Indeed, by definition of a perfect pairing, the $H_0$-linear map
\[
H_{c-i}\longrightarrow \Hom_{H_0}(H_i,H_c)
\]
sending $z\in H_{c-i}$ to the map given by left multiplication with $z$ is an isomorphism. Consider the commutative (check!) diagram
\[
\begin{tikzcd}
H_{c-i} \arrow{r}{\cong}\arrow{d}[swap]{\phi_{c-i}} & \Hom_{H_0}(H_i,H_c)\arrow{d}{\Hom_{H_0}(\phi_i^{-1},\phi_c)}\\
H_{c-i} \arrow{r}{\cong} & \Hom_{H_0}(H_i,H_c)
\end{tikzcd}
\]
and the analogous diagram with $\psi$ instead of $\phi$ commutes as well. By assumption the arrows on the right-hand side of the two diagrams agree, hence so do the arrows on the left-hand side, i.e., $\phi_{c-i}=\psi_{c-i}$.
\end{proof}

To round off this discussion we give an example of a Gorenstein ring for which the map on Koszul homology does depend on the lift, so that the preceding result seems to be the best one can hope for in general.

\begin{example}
\label{ex:gorenstein}
Consider the ring
\[
R\coloneqq \pos[\FF_2]{t^9,t^{10},t^{11},t^{13},t^{17}}\,.
\]
This ring is Gorenstein; see \cite[Example~3.8]{kustinmiller1982}. Take the free cover
$
f\colon F \to R$ represented by $\begin{pmatrix}t^9& t^{10}&t^{11}&t^{13}&t^{17}\end{pmatrix}
$.
Consider the automorphism of $K(f)$ that fixes $e_i$ for $i\neq 5$, and maps
\[
e_5\mapsto e_5+t^{10}e_2+t^{9}e_3\,.
\]
It is easy to verify that $[t^{19}e_1e_3+(t^{13}e_1+t^{11}e_3)e_5]$ is a class in $H(f)$ and that it is sent to itself plus \[
[t^{13}e_1+t^{11}e_3]\cdot [t^{10}e_2+t^{9}e_3]
\]
We verified with Macaulay2 and by hand that this product is not zero in $H(f)$.
\end{example}

\section{Filtrations}\label{sec:filtrations}
As in \cref{sec:framework}, let $(R,\fm,k)$ be a local ring and $f\colon F\to R$ a free cover of $\fm$. We consider an additional structure on the homology of the Koszul complex associated to $f$, which sheds further light on the problem we have been considering. 

The associated graded ring of $R$ with respect to its $\fm$-adic filtration is
\[
\gr_{\fm}(R) \coloneqq \bigoplus_{l\geq 0} {\fm^l}/{\fm^{l+1}}\,.
\]
Thus $\gr_{\fm}(R)_0=k$, and $\gr_{\fm}(R)$ is a graded $k$-algebra. One can extend the $\fm$-adic filtration on $R$ to one on $K(f)$ as follows; see, for example, \cite[\S3.8]{avramoviyengarmiller}.

\begin{definition} 
\label{de:serre-filtration}
Let $J$ be the kernel of the canonical augmentation $K(f)\to k$. This is a dg ideal in $K(f)$ with $J_0=\fm$. Setting $\filt{l}K(f) \coloneqq J^l$ for each integer $l\geq 0$ yields a decreasing filtration on $K(f)$, with 
\[
\filt{l}K_i(f) = \fm^{l-i} K_i(f)\,.
\]
Consider the associated graded object of this filtration:
\[
\gr K(f)\coloneqq \bigoplus_l J^l/J^{l+1}\,.
\]
It is easy to check that this is the Koszul complex on the map
\[
\gr(f)\colon \gr(F)(-1)\longrightarrow \gr_{\fm}(R)\,,
\]
where $\gr_{\fm}(F)(-1)$ denotes the free graded $\gr_{\fm}(R)$ module whose component in degree $l$ is $\fm^{l-1}F/\fm^l F$. We consider also the induced filtration on the homology of $K(f)$, namely the image
\[
\filt{l}H(f) \coloneqq \im(H(J^l)\longrightarrow H(f))\,.
\]
Thus a class $[z] \in H_i(f)$ is in $\filt{l}H(f)$ precisely when $z$ is homologous, in $K(f)$, to a cycle $w$ in $\fm^{l-i}K_i(f)$. Set
\[
 \gr H(f)\coloneqq \bigoplus_s \filt{s} H(f)/\filt{s+1}H(f)\,.
 \]
This is a graded $k$-algebra since the filtration of $H(f)$ is multiplicative. The filtration is also bounded, thus $H(f)$ and its associated graded $\gr H(f)$ are abstractly isomorphic as graded $k$-vector spaces.
\end{definition}

It is well-known that the algebras $H(f)$, $\gr H(f)$ and the Koszul homology of $\gr_\fm R$ are all different in general. We provide an example, with details.

\begin{example}
Consider the $\FF_2$-algebra 
\[
R\coloneqq \frac{\FF_2[x,y]}{(x^3+y^2, y^3)}
\]
and the free cover $f\colon R^2\xra{(x \; y)} R$. The $\FF_2$-module $H_1(f)$ has basis
\[
z_1\coloneqq [x^2e_1 + y e_2] \text{ and } z_2\coloneqq [y^2e_2]= [x^3e_2]\,.
\]
Since $R$ is a complete intersection $H(f)$ is the exterior algebra on $H_1(f)$; see \cref{rem:cirings}. On the other hand 
\[
\gr_\fm R\cong \frac{\FF_2[\overline{x},\overline{y}]}{(\overline{x}^6,\overline{x}^3\overline{y},\overline{y}^2)}
\]
as $\FF_2$-algebras, where $\overline{x}, \overline{y}$ are the classes of $x,y$, respectively, in $\fm/\fm^2$. Thus its Koszul homology is three dimensional in degree one, two dimensional in degree two, and products of positive degree elements are zero. In particular this algebra cannot be isomorphic to $H(f)$, nor to $\gr H(f)$, for the latter is isomorphic to $H(f)$ as a graded $\FF_2$-vector space. 

One can check $\gr H_2(f)=\filt{7}H_2(f)/0$ and $\gr H_1(f)=\filt{2}H_1(f)/\filt{3}H_1(f)\oplus \filt{4}H_1(f)/0$ which implies that products of positive degree elements of $\gr H(f)$ are zero as well. Thus $\gr H(f)$ and $H(f)$ cannot be isomorphic as $\FF_2$-algebras.
\end{example}

The constructions above are functorial in the following sense.

\begin{remark}
Taking the associated graded of a morphism
\[
(\phi,\alpha)\colon f\longrightarrow g
\]
in $\FreeCovers$ yields a map
\[(\gr \phi,\gr \alpha)\colon \gr(f)\longrightarrow \gr(g)
\]
in $\FreeCovers$. The induced morphism of dg algebras $K(\phi, \alpha)\colon K(f)\to K(g)$ maps $\filt{l}K(f)$ to $\filt{l}K(g)$, that it to say, it is a map of filtered complexes. We thus get a map of the associated graded objects 
\[
\gr K(\phi,\alpha)\colon \gr K(f) \longrightarrow \gr K(g) 
\]
and this map agrees with $K(\gr \phi,\gr \alpha)$. This is a $\gr(\alpha)$-equivariant morphism of dg algebras. In the same vein there is a map
\[
\gr(H(\phi,\alpha))\colon \gr H(f)\longrightarrow \gr H(g)\,,
\]
of graded rings.
\end{remark}

\begin{definition}
 Observe that the $k$-algebra $\gr_{\fm}(R)$ is generated by its degree one component; said otherwise, the natural map $\varepsilon\colon \mathrm{sym}_k(\fm/\fm^2)\to \gr_\fm(R)$ is surjective. The \emph{order} of $R$, denoted $\ord(R)$, is the minimal degree of a nonzero element in $\ker(\varepsilon)$. If $\ker(\varepsilon)\ne 0$, then $\ord(R)\ge 2$, and if $\ker(\varepsilon)=0$, equivalently, if $R$ is regular, the order is declared to be infinity.
\end{definition}

The following characterisation of orders is well-known; we give a proof for lack of a suitable reference. Recall that $K(f)$ is independent, up to an isomorphism of dg $R$-algebras, of the choice of a free cover of $\fm$; see \cite[\S1.6, page 52]{brunsherzog1998}.

\begin{lemma}
\label{le:ord}
With $f\colon F\to R$ a free cover of $\fm$, there is an equality
\[
\ord(R)= \sup\{l\mid \filt{l}{H_1(f)} = H_1(f)\}\,.
\]
\end{lemma}

\begin{proof}
One can assume $R$ is $\fm$-adically complete. Indeed, let $\iota\colon R\to \widehat R$ denote the $\fm$-adic completion of $R$. Then $\widehat{R}$ is a local ring with maximal ideal $\fm \widehat{R}$, and residue field $k$. The map $\iota$ induces an isomorphism of graded $k$-algebras
\[
\gr_{\fm}(R)\cong \gr_{\fm \widehat{R}}(\widehat{R})\,.
\]
In particular, $\ord(R) = \ord(\widehat{R})$. Moreover, $\widehat{f}\coloneqq \widehat{R}\otimes_Rf$ is a free cover of $\fm\widehat{R}$, there is an isomorphism $K(\widehat{f})\cong \widehat{R}\otimes_R K(f)$ of dg $\widehat{R}$-algebras and the map
\[
\iota\otimes_R K(f)\colon K(f)\longrightarrow K(\widehat{f})
\]
is a lift of $\iota$. It is a quasi-isomorphism, for $\fm\cdot H(f)=0$. Moreover this map, being a map of algebras, respects the filtrations on either side, and is a quasi-isomorphism on each level of the filtration. It follows that $H(f)$ and $H(\widehat{f})$ are isomorphic as filtered objects. It is also clear that $\gr(\iota\otimes_R K(f),\iota)$ is an isomorphism. Thus replacing $R$ by $\widehat{R}$, we assume $R$ is $\fm$-adically complete. 

Then there is a surjection $\pi\colon (Q,\fq,k)\to R$ where $Q$ is a regular local ring and the ideal $I\coloneqq \ker(\pi)$ satisfies $I\subseteq \fq^2$; see \cref{rem:degree-one-cycles}. We adopt the notation from \emph{op.\ cit.} If $I=0$, then $R$ is regular, and the desired result is clear, so we suppose $I$ is nonzero, so that $\ord(R)$ is finite. There is a commutative diagram of $k$-algebras 
\[
\begin{tikzcd}
\mathrm{sym}_k(\fq/\fq^2) \arrow{r}{\cong}\arrow{d}{\cong} & \gr_\fq(Q)\arrow{d}\\
\mathrm{sym}_k(\fm/\fm^2) \arrow{r} & \gr_\fm(R)
\end{tikzcd}.
\]
We fix a free cover $\widetilde{f}\colon \widetilde{F}\to Q$ of $\fq$; then $R\otimes_Q \widetilde{f}$ is a free cover of $\fm$. We can assume $f= R\otimes_Q\widetilde{f}$, so that $F = \widetilde{F}/I\widetilde{F}$. Let $\varphi\colon K(\widetilde{f})\to K(f)$ be the induced surjective map of dg algebras. The kernel of this map is $IK(\widetilde{f})$. 

Let $s$ be the order of $R$. It is easy to verify that $s=\sup\{j \mid I\subseteq \fq^j\}$; in particular, $I$ can be generated by elements $f_1,\dots,f_c$ with $f_i\in \fq^s$. From \cref{rem:degree-one-cycles} we get that the $k$-vector space $H_1(f)$ is spanned by cycles 
\[
z_i = \sum_j a_{ij}e_j \quad\text{for $1\le i\le c$},
\]
with $a_{ij} \in \fq^{s-1}$. Thus the $z_i$ are in $J^s$, where $J$ is the augmentation ideal of $K(f)$; see \cref{de:serre-filtration}. We conclude that $\filt{s}{H_1(f)}=H_1(f)$. 

It remains to show that if $\filt{l}{H_1(f)} = H_1(f)$ for some integer $l\ge 1$, then $I\subseteq \fq^l$. Given such an $l$, there must be cycles $w_i$ in $\fm^{l-1}K_1(f)$ that are homologous to the $z_i$. Since $\varphi$ is surjective, with kernel $IK(\widetilde{f})$, we can find elements $\widetilde{z}_i, \widetilde{w}_i$ in $K_1(\widetilde{f})$, with $\widetilde{w}_i$ in $\fq^{l-1}K_1(\widetilde{f})$, and $u_i$ in $K_2(\widetilde{f})$ such that for each $i$ one has
\[
\varphi(\widetilde{z}_i) = z_i\,, \quad \varphi(\widetilde{w}_i) = w_i\,, \quad\text{and}\quad du_i = \widetilde{z}_i-\widetilde{w}_i \quad \text{modulo $I K(\widetilde{f})$.}
\]
Moreover we can ensure that $d(\widetilde{z}_i) = f_i$. Since $d(K(\widetilde{f}))\subseteq \fq K(\widetilde{f})$ we deduce that $d(\widetilde{w}_i)\subseteq \fq^l K_0(\widetilde{f})=\fq^l$. Thus the relations above yield 
\[
f_i = d(\widetilde{z}_i) = d(\widetilde{w}_i) \in \fq^l Q\quad\text{for $1\le i\le c$.}
\]
This is the desired conclusion.
\end{proof}

In the light of \cref{le:ord}, in the result below one can take $s\coloneqq \ord(S)$ for the free cover $g\colon G\to S$, independent of the homomorphism $\delta$.

\begin{lemma} \label{lem:lemmaWithsAndFiltDegrees}
Fix $(\phi,\alpha)\colon f\to g$ in $\FreeCovers$, let $\delta\in \hom_\alpha(F,\ker g)$ and let $s\geq 0$ such that $[\delta(x)]\in \filt{s} H_1(g)$ for all $x\in F$. Then $H(\phi+\delta,\alpha)-H(\phi,\alpha)$ increases filtration degree by $s-1$, i.e., restricts to a map $\filt{l}H(f)\to \filt{l+s-1}H(g)$ for all $l\geq 0$. 
\end{lemma}

\begin{proof}
We choose a basis $e_1,\ldots,e_n$ of $F$. Without loss of generality, we can assume that $\delta(e_i)=0$ for all $i\ge 2$.

Let $z$ be a cycle whose homology class is in $\filt{l} H_i(f)$. It suffices to show that the difference $H_i(\phi+\delta,\alpha)(z)-H_i(\phi,\alpha)(z)$ lies in $\filt{l+s-1}H_i(g)$. Choose a representative of $z$ in $\filt{l}K_i(f)=\fm^{l-i}K_i(f)$ and write it in the form $e_1\cdot r+r'$, where $r,r'$ are in the subalgebra $L\subseteq K(f)$ generated by all $e_j$ with $j\neq 1$.
We have shown in the proof of \cref{lem:conntomult} that 
\[H_i(\phi+\delta,\alpha)(z)-H_i(\phi,\alpha)(z) = [\delta(e_1)]\cdot H_{i-1}(\phi,\alpha)([r])\]
From the direct sum decomposition $e_1\cdot L\oplus L$ it follows that $r\in \fm^{l-i}L_{i-1}=\filt{l-1}L_{i-1}$.
Thus one has $K_{i-1}(\phi,\alpha)(r)\in \filt{l-1}K_i(g)$. By assumption on $s$, we know that $[\delta(e_1)]\in \filt{s}H_1(g)$ and since the multiplication is compatible with the filtration, it follows that $[\delta(e_1)]\cdot H_{i-1}(\phi,\alpha)([r])\in \filt{l+s-1}H_{i}(g)$ as desired.
\end{proof}

\begin{corollary}\label{cor:filtrationargument}
If additionally $\filt{l}H_i(f)=H_i(f)$ and $\filt{l+s-1}H_i(g)=0$ in some degree $i\geq 0$, then 
$H_i(\phi+\delta,\alpha)=H_i(\phi,\alpha)$. 
\end{corollary}

\begin{proof}
The difference between $H_i(\phi+\delta,\alpha)$ and $H_i(\phi,\alpha)$ is an element of 
$\filt{l+s-1}H_i(g)$ and thus zero by assumption.
\end{proof}

\begin{example}\label{ex:filtrationCriterion}
Consider the $\FF_2$-algebra
\[
R\coloneqq \frac{\FF_2[x,y,z]}{(x^{98},y^{99},z^{100}, (x^{50} + y^{50}) z^{51})}
\]
and $f\colon R^3\to R$ the obvious free cover. Macaulay2 says that the Betti table of $R$ is
\begin{table}[H]
  \centering
  \begin{tabular}{r|c c c c c}
     & 0 &1 &2 & 3&\\
     \hline
     0: & 1 & - &- &- &\\
     97: & - & 1 & - &-& \\
     98: & - & 1 & - &- &\\
     99: & - & 1 & - &- &\\
     100: & - & 1 & - &- &\\
     148: & - & - & 1 & -&\\
     149: & - & - & 1 & -&\\
     195: & - & - & 1 & -&\\
     196: & - & - & 2 & -&\\
     197: & - & - & 1 & 1&\\
     244: & - & - & - & 1&\\
     245: & - & - & - & 1 &
  \end{tabular}
\end{table}
It follows that $\filt{150}H_2(f)=H_2(f)$ and $\filt{199}H_2(f)=0$. Since $\filt{98}H_1(f)=H_1(f)$, we apply \cref{cor:filtrationargument} with $p=98$ to conclude that $H_2(\phi,\alpha)$ is independent of the chosen lift $\phi$. 
\end{example}

\begin{corollary}\label{cor:graded}
The map
$\gr H(\phi,\alpha)$ is independent of the chosen lift $\phi$.
\end{corollary} 

\begin{proof}
We can assume $R$ is not regular, and then its order is at least $2$, so we can choose $s=2$ in \cref{lem:lemmaWithsAndFiltDegrees}. This gives the desired result.
\end{proof}

\begin{example}\label{ex:filtrationargument}
Consider the $\FF_2$-algebra
\[
R\coloneqq \frac{\FF_2[x,y,z,w]}{(x^2, y^2, z^2, w^2, yz-xw)}
\]
and $f$ the free cover of the maximal ideal given by the matrix $(x\;y\;z\;w)$. According to Macaulay2, the graded algebra $H(f)$ has 31 generators: 
\begin{align*}
 \text{degree } 1&:& &x e_1, y e_2, z e_3, w e_1 + z e_2, w e_4& \\
 \text{degree } 2&:& &x w e_1 e_2, z w e_1 e_2, x w e_1 e_3, z w e_2 e_3, y w e_2 e_3&\\
 \text{degree } 3&:& &z w e_1 e_2 e_3, y w e_1 e_2 e_3, x w e_1 e_2 e_3, x z e_1 e_2 e_3, x y e_1 e_2 e_3, zwe_1e_2e_4 &\\
 &&& y w e_1 e_2 e_4, x w e_1 e_2 e_4, x z e_1 e_2 e_4, x y e_1 e_2 e_4, z w e_1 e_3 e_4, y w e_1 e_3 e_4, &\\
 &&&  x w e_1 e_3 e_4, x z e_1 e_3 e_4, z w e_2 e_3 e_4, y w e_2 e_3 e_4& \\
 \text{degree } 4&:&& z w e_1 e_2 e_3 e_4, y w e_1 e_2 e_3 e_4, x w e_1 e_2 e_3 e_4, x z e_1 e_2 e_3 e_4, x y e_1 e_2 e_3 e_4&
\end{align*}
The multiplication $H_1\cdot H_i$ is trivial for $i\geq 2$, but $H_1\cdot H_1\neq 0$.

Note that $H_2(f)=\filt{4}H_2(f)$, and $\filt{5}H_2(f)=0$ since $\fm^3=0$. Thus $H_2(f)=\filt{4}H_2(f)/\filt{5}H_2(f)$ and it follows that every dg algebra automorphism of $K(f)$ induces the identity on $H_2(f)$ by \cref{cor:graded}. 
\end{example}

\section{Induced automorphisms of Koszul homology}\label{sec:inducedAutomorphismsKoszulHomology}

We have seen that in many cases the map in Koszul homology induced by some lift of a ring homomorphism is independent of the chosen lift. However, we also have seen some counterexamples. The images of one of the algebra generators under two different lifts just differ by a cycle of the Koszul complex. Thus it seems natural to ask whether the map on homology only depends on the homology classes of those cycles. This will be shown in \cref{prop:bdydontchangehomology}. This has some surprising consequences for the group of induced homology automorphisms; see \cref{thm:isgrouphomomorphism}.

Another consequence of \cref{prop:bdydontchangehomology} is that it reduces the cases that are required to check for finding out if the induced map $H(\phi,\alpha)$ is independent of the chosen lift $\phi$ of $\alpha$. 
\begin{proposition}
\label{prop:bdydontchangehomology}
Let $(\phi,\alpha)$ be a morphism from $f\colon F\to R$ to $g\colon G\to S$ in $\FreeCovers$. If $\delta\in \hom_\alpha(F,B)$, where $B\subseteq \ker(g)$ are the boundaries in the Koszul complex, then $K(\phi+\delta,\alpha)$ and $K(\phi,\alpha)$ are homotopic as maps of chain complexes and thus 
\[
H(\phi+\delta,\alpha)=H(\phi,\alpha)\,.
\]
\end{proposition}

\begin{proof} 
Let $e_1,\ldots,e_n$ be a basis of $F=K_1(f)$. We may assume without loss of generality that $\delta$ sends all basis vectors but $e_i$ to zero. 

By assumption on $\delta$, there exists a chain $s\in K_2(g)$ such that
\[
\delta(e_i) = ds
.\] Let $h\colon K_*(f)\to K_{*+1}(g)$ be the map of degree one sending a basis element $e_{i_1}\wedge \ldots \wedge e_{i_m}$ of $K_m(f)$ to $0$ if $i_l\neq i$ for all $1\leq l\leq m$ and to
\[
 (-1)^{l-1}\phi(e_{i_1})\wedge \ldots \wedge \phi(e_{i_{l-1}}) \wedge s \wedge \phi(e_{i_{l+1}}) \wedge \ldots \wedge \phi(e_{i_m})
\]
if $i_l=i$ for some $l$. We will show that $h$ defines a chain homotopy between $K(\phi+\delta,\alpha)$ and $K(\phi,\alpha)$. The difference
$(K(\phi+\delta,\alpha)-K(\phi,\alpha))(e_{i_1}\wedge \ldots \wedge e_{i_m})$ ,i.e,
\[ (\phi+\delta)(e_{i_1})\wedge\ldots \wedge (\phi+\delta)(e_{i_m}) - \phi(e_{i_1})\wedge\ldots \wedge \phi(e_{i_m})\]
is zero if $i_l\neq i$ for all $l$ by assumption on $\delta$, and hence equals $(dh+hd)(e_{i_1}\wedge \ldots \wedge e_{i_m})$ in this case as desired. If $i_l=i$ for some $l$, then the difference above equals
\[\phi(e_{i_1})\wedge \ldots \wedge \phi(e_{i_{l-1}}) \wedge ds \wedge \phi(e_{i_{l+1}}) \wedge \ldots \wedge \phi(e_{i_m})
\]
by assumption on $\delta$ which by a straightforward computation agrees again with
$(dh+hd)(e_{i_1}\wedge \ldots \wedge e_{i_m})$.

Thus $K(\phi+\delta,\alpha)$ and $K(\phi,\alpha)$ are chain homotopic via $h$.
\end{proof}
This result is a first step for checking with a computer program. If $e_1,\ldots, e_n$ is a basis of $F$, then an arbitrary lift of $\alpha$ sends a basis element $e_i$ to $\phi(e_i)+\delta_i$ for some $\delta_i\in \ker g$. Thus a priori, the program needs to run through all tuples $(\delta_1,\ldots,\delta_n)$ of elements of $\ker g$ to check if the induced map in homology is independent of the lift $\phi$. By \cref{prop:bdydontchangehomology}, it suffices to choose a representative of each homology class in $H_1(g)$ and run through $n$-tuples of such representatives.

We will use the following lemma to further reduce the number of cases that need to be checked and to establish structural properties of the group of automorphisms in Koszul homology that are induced by automorphisms of the Koszul dga.

\begin{lemma}\label{lem:composablemorphisms} Let two composable morphisms in $\FreeCovers$ be given
\[
\begin{tikzcd}
F \arrow{d}[swap]{f} \arrow{r}{\phi} & F'\arrow{d}{f'}\arrow{r}{\phi'}&F''\arrow{d}{f''}\\
R \arrow{r}[swap]{\alpha} & R' \arrow{r}[swap]{\alpha'}&R''
\end{tikzcd}.
\]
 Different lifts are given by 
$\phi+\delta$ for $\delta\in \hom_\alpha(F,\ker(f'))$ and by $\phi'+\delta'$ for $\delta'\in \hom_{\alpha'}(F',\ker(f''))$. Then the following maps are homotpic:
\[
(\phi'+\delta')\circ(\phi+\delta) \qquad\text{and}\qquad \phi'\circ \phi +\phi'\circ \delta+\delta'\circ \phi\,.
\]
\end{lemma}

\begin{proof}
We denote the maximal ideals of $R$,$R'$ and $R''$ by $\fm,\fm'$ and $\fm''$, respectively. The difference of the two maps in question is $\delta'\circ \delta$. As $\im(\delta)\subseteq \ker(f')$ and $\ker(f')\subseteq \fm'F'$, it follows that
\[\im(\delta'\circ \delta) \subseteq
\delta'(\ker(f'))\subseteq 
\delta'(\fm'F')\subseteq \alpha'(\fm') \delta'(F')
\subseteq \mathfrak{m''} \ker(f'').
\]
Since $\fm''\cdot H(f'')=0$, all elements in $\fm'' \ker(f'')$ are boundaries. Thus \cref{prop:bdydontchangehomology} gives the desired result.
\end{proof}

We are interested in the group $\aut(K(f))$ of automorphisms of the dga $K(f)$ and the induced automorphisms on Koszul homology. There is a bijection
\[\hom_{\id_R}(F,\ker(f))\rightarrow \aut(K(f)), \quad c\mapsto K(\id_F+\delta,\id_R).\] The source is an abelian group under addition and the target is a group under composition. In general this map is not a group homomorphism. 
Nevertheless, we will show that it is compatible up to homotopy and hence a group homomorphism after passing to homology.
\begin{theorem}\label{thm:isgrouphomomorphism}
The map
\[\hom_{\id_R}(F,H_1(f))\rightarrow \aut(H(f)), \quad [\delta]\mapsto H(\id_F+\delta,\id_R),
\]
is a group homomorphism. Thus any two automorphisms in the image commute. Therefore when the residue field of $R$ is of positive characteristic $p$ any nontrivial automorphism in the image has order $p$.
\end{theorem}

\begin{proof}
The map in question is well-defined, since adding boundaries to the representative does not affect the map on the right-hand side by \cref{prop:bdydontchangehomology}. We apply \cref{lem:composablemorphisms} with $F=F'=F''$ and $R=R'=R''$ and $\phi=\phi'=\id$ and $\alpha=\alpha'=\id$. It follows that $(\id+\delta')\circ (\id+\delta)$ is homotopic to $(\id+\delta'+\delta)$. In particular, we get 
\[
H(\id+\delta',\id)\circ H(\id+\delta,\id)=H(\id+\delta'+\delta,\id)
\]
which shows that the map is a homomorphism of groups as desired. If the characteristic of $k=R/\fm$ is $p>0$, then any element in the abelian group $\hom_{\id_R}(F,H_1(f))$ is trivial or has order $p$, hence so does its image.
\end{proof}

\cref{thm:isgrouphomomorphism} has the following computational consequence.
\begin{remark}\label{rem:reducecasesforprogram}
If the residue field $k=R/\fm$ is finite and $\{z_l\}_l$ are cycles representing a basis for $H_1(f)$, then it suffices to check for a fixed basis $e_1,\ldots, e_n$ of the free cover of $f$ if for each fixed $i$ the maps
\[
  e_i\mapsto e_i + z_l, \quad e_j\mapsto e_j \text{ for }j\neq i,
\]
induce the identity on homology to conclude whether every automorphism of $K(f)$ induces the identity on homology.
\end{remark}

Another implication of \cref{thm:isgrouphomomorphism} is that every automorphism of $H(f)$ that is induced by an automorphism of the dga $K(f)$ has finite order if the characteristic of the ring $R$ is nonzero. Without such a constraint on $k$, there can be automorphisms of infinite order. Here is an example.

\begin{example}
\label{ex:smllWithQ}
Consider the ring 
\[
R\coloneqq \frac{\QQ[x,y,z]}{(x^2,xy,y^2,z^2)}\,.
\]
Consider the free cover $f\colon R^3\xrightarrow{(x\; y\;z)} R$ of the maximal ideal $(x,y,z)$ of $R$. The Betti table of $R$ is
\begin{table}[H]
  \centering
  \begin{tabular}{r|c c c c c} 
     & 0 &1 &2 & 3&\\
     \hline
     0: & 1 & - &- &- &\\
     1: & - & 4 & 2 &-& \\
     2: & - & - & 3 &2 &
  \end{tabular}
\end{table}
Fix $\lambda\in \QQ$ and consider the map $\phi \colon R^3\to R^3$ where
\[
e_1\mapsto e_1+\lambda ze_3
\]
and $e_2,e_3$ are fixed. In $H(f)$ consider the classes 
\[
z_1 = [xe_1], z_2=[ye_2], z_3 = [ye_1], z_4= [ze_3], z_5= [ye_1e_2], z_6=[xe_1e_2]
\]
From Macaulay2 we get that
\[
H(f)=\frac{\QQ[z_1,\ldots, z_6]}{(z_2 z_3, z_1 z_3, z_1 z_2, z_3 z_6, z_2 z_6, z_1 z_6, z_3 z_5, z_2 z_5, z_1 z_5, z_6^2, z_5 z_6, z_5^2)}.
\]
It is easy to check that 
\[
H(\phi,\id)(z_5)=[y\phi(e_1)e_2]=z_5+\lambda[zye_3e_2]=z_5+\lambda z_4z_2,
\]
and that $z_4z_2$ is not zero. Thus the map $\QQ\to \aut(H(f))$ that sends $\lambda$ to the automorphism above defines an embedding, giving automorphisms of infinite order.
\end{example}

\bibliographystyle{amsalpha}
\bibliography{bibl}
\end{document}